\documentclass{amsart}

\usepackage{a4wide}

\usepackage{amssymb}
\usepackage[all]{xypic}

\newcounter{nootje}
\renewcommand\check[1] 
{}

\newcommand\QQ{\mathbb{Q}}
\newcommand\sL{\mathcal{L}}

\newcommand\gV{\mathcal{V}}
\newcommand\OO{\mathcal{O}}
\newcommand\PP{\mathbb{P}}
\renewcommand\AA{\mathbb{A}}
\newcommand\RR{\mathbb{R}}
\newcommand\ZZ{\mathbb{Z}}
\newcommand\kbar{{\overline{k}}}
\newcommand\isom{\cong}
\newcommand\Pic{{\rm Pic}}
\newcommand\tilder{}

\newtheorem{theorem}{Theorem}[section]
\newtheorem{proposition}[theorem]{Proposition}
\newtheorem{conjecture}[theorem]{Conjecture}
\newtheorem{condition}[theorem]{Condition}
\newtheorem{lemma}[theorem]{Lemma}
\newtheorem{definition}[theorem]{Definition}
\newtheorem{remark}[theorem]{Remark}

\begin{document}

\begin{center}
{\Large Ronald van Luijk} \\
{\large Density of rational points on elliptic surfaces}\\
\bigskip

\begin{minipage}{5in}
\it Suppose $V$ is a surface over a number field $k$ that admits two elliptic fibrations. We show that for each integer $d$ there exists an explicitly computable closed subset $Z$ of $V$, not equal to $V$, such that for each field extension $K$ of $k$ of degree at most $d$ over the field of rational numbers, the set $V(K)$ is Zariski dense as soon as it contains any point outside $Z$. We also present a version of this statement that is universal over certain twists of $V$ and over all extensions of $k$. 
This generalizes a result of Swinnerton-Dyer, as well as previous work of Logan, McKinnon, and the author.
\end{minipage}
\end{center}

\section{Introduction}

Logan, McKinnon, and the author proved the following theorem in \cite{LMvL}.

\begin{theorem}\label{diagonal}
Let $V$ be a diagonal quartic surface in $\PP_\QQ^3$, given by $ax^4+by^4+cz^4+dw^4=0$ for some 
coefficients $a,b,c,d \in \QQ^*$ whose product $abcd$ is a square. If $V$ contains a rational point
$P = [x_0: y_0: z_0: w_0]$ with $x_0y_0z_0w_0 \neq 0$ that is not contained in one of the $48$ lines
on $V$, then the set  $V(\QQ)$ of rational points on $V$ is Zariski dense in $V$, as well as dense in the 
real analytic topology on $V(\RR)$.  
\end{theorem}

The proof relies on the two elliptic fibrations that exist generically on diagonal quartic surfaces whose coefficients have square product.  
Swinnerton-Dyer \cite{SD} then showed that in much higher generality, namely for any K3 surface $V$ over $\QQ$ with at least two elliptic fibrations, there exists an explicitly computable Zariski closed subset $Z \subsetneq V$, such that if $V$ contains a rational point outside $Z$, then $V(\QQ)$ is Zariski dense in $V$; he mentions that similar arguments work over any number field. Here, and in the remainder of this paper, {\em explicitly computable} means that there is an algorithm that takes as input equations for both the surface $V$ and the two fibrations, and that gives as output equations for the closed subset $Z$. In that same paper \cite{SD}, Swinnerton-Dyer produces some nice results about density of $V(\QQ)$ in the real analytic and $p$-adic topologies as well. He also gives a cleaner proof of Theorem \ref{diagonal}, based on explicit formulas taken from \cite{SDdiag}.

Inspired by Swinnerton-Dyer's generalization, we similarly generalize another result from \cite{LMvL}, namely a version of Theorem \ref{diagonal} over number fields that is in some sense uniform over finite extensions. 
The only topology we deal with is the Zariski topology.  
The main arguments are essentially the same as the ones in \cite{LMvL}. Those were phrased differently from 
Swinnerton-Dyer's \cite{SD} in the sense that 
where the paper \cite{LMvL} uses an endomorphism $\alpha \colon F \to F$ of a genus-one curve $F$, Swinnerton-Dyer uses instead the associated covering $\chi\colon F \to J(F), P \mapsto (P) - (\alpha(P))$ of the Jacobian $J(F)$ of $F$, so that $\alpha(P)$ is the translation of $P$ by $-\chi(P)$. In this paper we will use both of the equivalent points of view. 
Arguments similar to ours are also used by Bogomolov and Tschinkel \cite{elliptic,ellipticK3}, and Harris and Tschinkel \cite{quartic} in the setting of potential density. 

\section{Setting and main theorems}\label{settings}

Let $k$ be a number field and let $\kbar$ be an algebraic closure of $k$. Let $V$ be a smooth projective surface over $k$. For $i =1,2$, let $f_i \colon V \to C_i$  be an elliptic fibration over $k$ to a curve $C_i$, and let $\gV_i$ be the generic fiber of $f_i$. 
We do not assume that the fibrations have a section, nor that they be minimal.  
We {\bf do assume}  that the fibrations are different in the sense that no fiber of $f_1$ is algebraically equivalent to a fiber of $f_2$; this is equivalent to the irreducible fibers of either one of the fibrations being horizontal curves with respect to the other fibration.

For $i=1,2$, let $\alpha_i \colon V \dashrightarrow V$ be a rational map that respects $f_i$. 
Then the map $\alpha_i$ is well defined on all smooth fibers of $f_i$. 
\check{mention why? N\'eron-models, or easier? restriction to a smooth fiber extends to the whole fiber}
Let $\alpha_i^\circ \colon \gV_i \to \gV_i$ be the restriction of $\alpha_i$ to the generic fiber $\gV_i$. 
Let $J(\gV_i)$ denote the Jacobian of $\gV_i$ and let $\chi_i \colon \gV_i \to J(\gV_i)$ be the map that sends $P$ 
to $(P) - (\alpha_i^\circ(P))$. 
We {\bf assume} that the map $\chi_i$ is not constant for $i=1,2$ and let $M_i$ denote the degree of $\chi_i$. 
In other words, the restriction $\alpha_i^\circ$ of $\alpha_i$ to the generic fiber $\gV_i$ is not merely 
translation by an element of the Jacobian $J(\gV_i)$. This is then automatically also the case for the 
restriction of $\alpha_i$ to all smooth fibers.  
In Remark \ref{alphafromL} we will see that rational maps such as $\alpha_1$ and $\alpha_2$ always exist.

\begin{definition}\label{deftwist}
A twist of the quintuple $(V,f_1,f_2,\alpha_1,\alpha_2)$ is a quintuple $(W,g_1,g_2,\beta_1,\beta_2)$, where $W$ is a variety, where $\beta_i\colon W \dashrightarrow W$ is a rational map respecting the fibration $g_i \colon W \to D_i$ over a curve $D_i$ for $i=1,2$, with all objects defined over $k$ and such that over $\kbar$ there are isomorphisms $\psi_i \colon (D_i)_\kbar \to (C_i)_\kbar$ and $\varphi \colon W_\kbar \to V_\kbar$, making the diagrams 
$$
\xymatrix{
& W_\kbar \ar'[d][dd]^\isom_{\varphi}  \ar[rr]^(0.4){g_i} && (D_i)_\kbar \ar[dd]^\isom_{\psi_i} \\
 W_\kbar \ar[dd]^\isom_{\varphi}  \ar[rr]^(0.65){g_i} \ar[ru]^{\beta_i}&& (D_i)_\kbar \ar[dd]^(0.3)\isom_(0.3){\psi_i} \ar@{=}[ur] \\
& V_\kbar  \ar'[r]_(0.5){f_i}[rr] && (C_i)_\kbar \\
V_\kbar  \ar[rr]_(0.45){f_i} \ar[ru]_{\alpha_i}&& (C_i)_\kbar \ar@{=}[ur] 
}
$$
commutative for $i=1,2$.
\end{definition}

By abuse of language, when we talk about a twist $(W,g_1,g_2)$ of $(V,f_1,f_2)$, or even a twist $W$ of $V$, we implicitly assume the existence of rational maps $\beta_1,\beta_2 \colon W \dashrightarrow W$, as well as morphisms $g_1,g_2$ in the latter case, for which $(W,g_1,g_2,\beta_1,\beta_2)$ is a twist of $(V,f_1,f_2,\alpha_1,\alpha_2)$. 
If we talk about an isomorphism $\varphi\colon W_\kbar \to V_\kbar$ corresponding to a twist $W$ of $V$, then we mean {\em some} isomorphism $\varphi$ for which there also exist $\psi_1$ and $\psi_2$ as in Definition \ref{deftwist}. 
Our first main result is the following.

\begin{theorem}\label{mainone}
For each integer $d$ there exists an explicitly computable closed subset $Z \subsetneq V$ such that for each field extension $K$ of $k$ of degree at most $d$ over $\QQ$ and for each twist $W$ of $V$, with corresponding isomorphism $\varphi\colon W_\kbar \to V_\kbar$, the set $W(K)$ is Zariski dense in $W$ as soon as it contains any point outside $\varphi^{-1}(Z)$. 
\end{theorem}

Theorem \ref{mainone} implies Swinnerton-Dyer's Theorem 1 in \cite{SD} mentioned above, and is stronger in the sense that it is uniform over all twists of $V$ as well as over all finite extensions of bounded degree.  

For $i=1,2$, let the $j$-map $j_i \colon C_i \to \PP^1$ be given by $j_i(t) = j\big(f_i^{-1}(t)\big)$, the $j$-invariant of the (Jacobian of the) genus-one fiber $f_i^{-1}(t)$. If the map $j_i$ is nonconstant, then we let $d_i$ be its degree, otherwise we set $d_i = \infty$.  
If the $j$-maps $j_1$ and $j_2$ are both nonconstant, then there is a pseudo-uniform version of Theorem \ref{mainone} over {\em all} finite extensions of $k$ in the sense that for larger extensions the closed subset $Z$ only needs to be enlarged by a finite number of points. We will give a bound for this number that depends only on the field extension, the degrees $d_1$ and $d_2$, and the degrees $M_1$ and $M_2$, but our methods do not allow such a bound to be computed explicitly.   
Our second main result is the following.

\begin{theorem}\label{maintwo}
Assume the $j$-maps $j_1$ and $j_2$ are nonconstant. Then 
there exists an explicitly computable closed subset $Z \subsetneq V$ such that for each finite extension $K$ of $k$ there is an integer $n$ that only depends on $K$, such that for each twist $W$ of $V$, with corresponding isomorphism $\varphi\colon W_\kbar \to V_\kbar$, 
the set $W(K)$ is Zariski dense in $W$ as soon as it contains 
more than $n\cdot \min (d_1M_1,d_2M_2)$ points outside $\varphi^{-1}(Z)$. 
\end{theorem}

\begin{remark}\label{alphafromL}
Examples of rational maps $\alpha_1$ and $\alpha_2$ can be constructed as follows. Take  $i\in \{1,2\}$ and take a line bundle $\sL_i$ on $V$ or, more generally, a line bundle $\sL_i$ on $V_{\kbar}$ whose isomorphism class is defined over $k$.  Let $m_i$ denote the degree of the restriction $(\sL_i)_F$ of $\sL_i$ to any smooth fiber $F$ of $f_i$. Define $\tilder{\alpha}_i \colon V \dashrightarrow V$ by $\tilder{\alpha}_i(P)=R$ for the unique point $R$ on the fiber $F = f_i^{-1}(f_i(P))$ of $f_i$ through $P$ for which $\OO_F(R)$ is isomorphic to the degree-one bundle $\sL_F \otimes \OO_F\big((1-m_i)P\big)$. In this case also the Theorem of Riemann-Roch \cite[Theorem IV.1.3]{hag} shows that the map $\tilder{\alpha}_i$ is well defined on the smooth fibers of $f_i$.  

The map $\chi_i \colon \gV_i \to J(\gV_i)$ is in this case induced by the map $\gV_i \to \Pic^0 \gV_i, P \mapsto \OO_{\gV_i}(m_iP) \otimes \sL_{\gV_i}^{-1}$ and is the $m_i$-covering of $J(\gV_i)$ corresponding to what Swinnerton-Dyer calls $\psi$ (see \cite{SD}). 
The assumption that the map $\chi_i$ not be constant  is equivalent to $m_i$ being nonzero. 
In this example the degree of $\chi_i$ equals $M_i = m_i^2$. 


If the endomorphism ring of the generic fiber $\gV_i$ is just $\ZZ$, then all rational maps respecting $f_i$ are of this form. These rational maps are a direct generalization of the maps $e_1$ and $e_2$  used in \cite{LMvL}, where we had $\sL_1 = \sL_2 = \OO_V(1)$ and $m_1=m_2=4$, cf. \cite[Remark 2.15]{LMvL}. Also on the diagonal quartic surface given by $x^4+y^4+z^4 - t^4=0$, studied by Elkies \cite{elkies}, where the product of the coefficients is not a square, there exist 
two elliptic fibrations whose fibers are intersections of two quadrics, so again we could take $\sL_1 = \sL_2 = \OO_V(1)$ and $m_1=m_2=4$. 


Suppose the line bundles $\sL_1$ and $\sL_2$ induce rational maps $\tilder{\alpha}_1$ and $\tilder{\alpha}_2$ respectively. 
Also assume that for $i=1,2$ we have fibrations $g_i \colon W \to D_i$ of a variety $W$ to a curve $D_i$ and isomorphisms 
$\psi_i \colon (D_i)_\kbar \to (C_i)_\kbar$ and $\varphi \colon W_\kbar \to V_\kbar$, making the front face of the diagram of Definition \ref{deftwist} commutative. If for $i=1,2$ the isomorphism class of $\varphi^*(\sL_i)$ is defined over $k$, then we can associate a rational map $\tilder{\beta}_i \colon W \dashrightarrow W$ to it to obtain a twist 
$(W,g_1,g_2,\tilder{\beta}_1,\tilder{\beta}_2)$ of $(V,f_1,f_2,\tilder{\alpha}_1,\tilder{\alpha}_2)$. 
\end{remark}

Surfaces of Kodaira dimension $1$ admit a unique elliptic fibration \cite[Proposition IX.3]{beauville}, while those of Kodaira dimension $2$ do not admit any. 
This means that our results are constrained to surfaces of Kodaira dimension $-1$ and $0$. 

Of course there exist abelian surfaces containing only finitely many rational points over some number field. But there is no K3 surface over a number field that is known to contain only a finite, positive number of rational points. It may therefore be the case that Theorem \ref{mainone} is true for K3 surfaces even if we take $Z = \emptyset$. 
An interesting family of examples in this context is given by the diagonal quartic surfaces of the form $x^4-y^4 = t(z^4-w^4)$ for some rational number $t \in \QQ$. They contain a trivial point $[1:1:1:1]$, which would imply that the set of rational points is dense. For all $t$ with numerator and denominator at most $100$ this has been verified using Theorem \ref{diagonal}. This leads to the following conjecture. 

\begin{conjecture}
Every number can be written as the ratio of two differences of fourth powers. 
\end{conjecture}

In the next section we will state and prove explicit versions of Theorem \ref{mainone} and \ref{maintwo}. Those also allow one to easily check whether a given point is contained in the mentioned subset $Z$. 

\section{Explicit subsets}

The proof of Theorems \ref{mainone} and \ref{maintwo} relies on an explicit version of Merel's Theorem \cite[Corollaire]{merel}, which bounds the torsion subgroup of the Mordell-Weil group of any elliptic curve over a number field. Oesterl\'e  sharpened Merel's original explicit bound on possible prime orders. He showed 
\check{oesterl\'e published anywhere??}
 that if  $E$ is an elliptic curve over a number field $K$ of degree $d$ over $\QQ$ and the Mordell-Weil group $E(K)$ contains a point of prime order $p$, then we have $p \leq (1+3^{d/2})^2$; 
Parent \cite[Th\'eor\`eme 1.2]{parent} shows that if $E(K)$ contains a point of prime power order $p^n$ with $p$ prime, then we have 
\begin{equation}\label{parentbounds}
p^n \leq  \left\{
\begin{array}{ll}
65(3^d-1)(2d)^6 & (p \neq 2,3), \\
65(5^d-1)(2d)^6 & (p =3), \\
129(3^d-1)(3d)^6 & (p =2). 
\end{array}
\right.
\end{equation}
This is summarized in the following theorem.

\begin{theorem}[Merel, Oesterl\'e, Parent]\label{merel}
The torsion subgroup of an elliptic curve over a number field of degree at most $d$ is isomorphic 
to a subgroup of $\ZZ/B\ZZ \times \ZZ/B\ZZ$ with 
\begin{equation}\label{globalbound}
B = \prod_{p \leq (1+3^{d/2})^2} p^{n_p},
\end{equation}
where the product ranges over primes $p$ and where $p^{n_p}$ is the largest power of $p$ satisfying \eqref{parentbounds}. 
\end{theorem}


\begin{definition}\label{defT}
For $i=1,2$ and any positive integer $r$, we let $T_{i,r}$ denote the closure of the locus of all points $P \in V$, for which the fiber $F = f_i^{-1}(f_i(P))$ is smooth and for which the divisor $\big(\alpha_i(P)\big) - (P)$ on $F$ has exact order $r$ in the Jacobian of $F$. 
\end{definition}

The map from the smooth fiber $F$ mentioned in Definition \ref{defT} to its Jacobian, given by $P \mapsto 
\big(\alpha_i(P)\big) - (P)$, is not constant by assumption (in fact of degree $M_i$), so  the divisor $\big(\alpha_i(P)\big) - (P)$ is only of order $r$ for finitely many points $P$ on $F$, and the set $T_{i,r}$ does not contain $F$. It follows that $T_{i,r}$ does not contain any irreducible components of fibers of $f_i$ for all positive integers $r$.  

Note that $T_{i,r}$ is explicitly computable as follows. Take the generic point $\eta$ on the generic fiber $\gV_i$. Then $\alpha_i(\eta)$ is another point on $\gV_i$. After bringing the generic fiber $\gV_i$ with distinguished point $\eta$ into Weierstrass form, the point $\alpha_i(\eta)$ corresponds with a point that we can equate to the $r$-torsion points, which we can find with the $r$-division polynomials. This gives equations for those $P$ for which $\alpha_i(P)$ is an $r$-torsion point on its smooth fiber with $P$ as distinguished neutral element; this is equivalent to the condition for $T_{i,r}$.


For $i=1,2$, we let $S_i$ denote the union of the singular fibers of $f_i$ and for each integer $x$ we set 
$$
T_i(x) = \bigcup_{1 \leq r \leq x} T_{i,r}.
$$
It is not hard to prove Theorem \ref{mainone} by showing that for any twist $W$ of $V$ with corresponding isomorphism $\varphi \colon W_\kbar \to V_\kbar$, and for any finite extension $K$ of $k$ of degree $d$ over $\QQ$, with $B$ as in \eqref{globalbound}, the set $W(K)$ is dense in $W$ as soon as it contains a point outside the set 
$\varphi^{-1}(S_1 \cup S_2 \cup T_1(B) \cup T_2(B))$. 
We will show in Proposition \ref{efficient} that the same conclusion holds when we replace this set by a much smaller one. 
This stronger statement, however, requires a little more care to prove. 

To avoid having to choose a twist $W$ of $V$ in almost every statement of the remainder of this section, we now fix a twist $(W,g_1,g_2,\beta_1,\beta_2)$ of $(V,f_1,f_2,\alpha_1,\alpha_2)$, knowing that everything that will be proved for $W$, in fact holds for every twist. Let $D_1$ and $D_2$ be the base curves of the fibrations $g_1$ and $g_2$ respectively. 
Let  $\psi_i \colon (D_i)_\kbar \to (C_i)_\kbar$, for $i=1,2$, and $\varphi \colon W_\kbar \to V_\kbar$ be isomorphisms making the diagrams
of Definition \ref{deftwist} commute. 

\begin{condition}\label{Xi}
Let $x$ be an integer and $K$ an extension of $k$. 
We say that a point $P \in W(K)$ satisfies $\Xi_i(x)$, for $i=1,2$, if the fiber $F= g_i^{-1}(g_i(P))$ of $g_i$ through $P$ is smooth and 
for every integer $r>x$ the divisor class of $\big(\beta_i(P)\big)-(P)$ does not have order $r$ in the Jacobian of $F$. 
\end{condition}

Note that a point $P\in W(K)$ satisfies $\Xi_i(x)$ if and only if $P$ lies on a smooth fiber and we have $\varphi(P) \not \in \bigcup_{r>x} T_{i,r}$. 

\begin{definition}
Suppose $\{i,j\} = \{1,2\}$ and let $x$ be an integer. Then we let $Z_i(x)$ be the union of $T_i(x)$ and the singular points of singular fibers of $f_i$.
\end{definition}

\begin{lemma}\label{fiberinfinite}
Suppose $i \in \{1,2\}$, let $K$ be any field extension of $k$, and let $x$ be a positive integer.  
Suppose that $W(K)$ contains a point $P$ outside $\varphi^{-1}(Z_i(x))$ and let $F= g_i^{-1}(g_i(P))$ be the fiber of $g_i$ through $P$. 
Then the following statements hold. 
\begin{enumerate}
\item If $F$ is singular and $C\subset F$ is an irreducible component of $F$ containing $P$, then $C(K)$ is infinite. 
\item 
%
If $P$ satisfies $\Xi_i(x)$,
then the set $F(K)$ is infinite.  
\end{enumerate}
\end{lemma}
\begin{proof}
If $F$ is a singular fiber, then $P$ is a smooth point on $F$, so $C$ is the unique component of $F$ containing $P$, and therefore $C$ is also defined over $K$; 
since the genus of $C$ equals $0$ in this case, we find that $C$ is birational over $K$ to $\PP^1$, so $C(K)$ is infinite indeed. 
Suppose $F$ is smooth and $P \in F(K)$ satisfies $\Xi_i(x)$. As the fiber $F$ has a $K$-point, it is isomorphic to its Jacobian $J=J(F)$, so it suffices to show that the divisor $D = \big(\beta_i(P)\big)-(P) \in J(K)$ has infinite order. This is a geometric statement, so we assume $(W,g_1,g_2,\beta_1,\beta_2)=(V,f_1,f_2,\alpha_1,\alpha_2)$ without loss of generality. 
The divisor $D$ does not have order $r$ in $J(K)$ for any $r\leq x$ per definition of $Z_i(x)$. As it also 
does not have order $r$ for any integer $r>x$ per assumption of $\Xi_i(x)$, we conclude that it has infinite order, which finishes the proof.
\end{proof}

An immediate consequence of Lemma \ref{fiberinfinite} is the following lemma. 

\begin{lemma}\label{dense}
Suppose $i \in \{1,2\}$, let $K$ be any field extension of $k$, and let $x$ be a positive integer.  
Let $C \subset W$ be a horizontal curve with respect to $g_i$ that is not contained in $\varphi^{-1}(T_i(x))$ and for which $C(K)$ is infinite.  If only finitely many points $P \in W(K)$ outside $\varphi^{-1}(S_i \cup T_i(x))$
do not satisfy $\Xi_i(x)$, then $W(K)$ is Zariski dense in $W$.  
\end{lemma}
\begin{proof}
The curve $C$ intersects $\varphi^{-1}(T_i(x))$ and each fiber of $g_i$ in only finitely many points. 
Therefore there are infinitely many smooth fibers with a point in $C(K)$ and for all but finitely many of those fibers there is a such point that is not contained in $\varphi^{-1}(T_i(x))$ and does satisfy $\Xi_i(x)$. Since $\varphi^{-1}(T_i(x))$ and $\varphi^{-1}(Z_i(x))$ only differ in singular fibers, Lemma \ref{fiberinfinite} implies that there are infinitely many fibers of $g_i$ with an irreducible component that contains infinitely many $K$-rational points, so $W(K)$ is Zariski dense.
\end{proof}

\begin{definition}
For any integer $x$ we let 
$\mathcal{C}(x)$ denote the collection of all irreducible components of fibers of $f_1$ or $f_2$ that are contained in 
$(S_1 \cap S_2) \cup T_1(x) \cup T_2(x)$ and we set 
$$
Z_0(x) = \bigcup_{C \in \mathcal{C}(x)} C.
$$
\end{definition}

Note that $T_i(x)$ does not contain any components of fibers of $f_i$ for $i=1,2$, so $\mathcal{C}(x)$ consists of irreducible curves  that for both fibrations are contained in a singular fiber and of components of any fiber of $f_1$ that are contained in $T_2(x)$ or vice versa.

\begin{proposition}\label{efficient}
Let $K$ be a finite extension of $k$ of degree at most $d$ over $\QQ$ and let $B$ be as in \eqref{globalbound}. If $W(K)$ contains a point outside $\varphi^{-1}(Z)$ for $Z = Z_0(B) \cup (Z_1(B) \cap Z_2(B))$, then $W(K)$ is dense in $W$. 
\end{proposition}
\begin{proof}
Suppose $P \in W(K)$ is a point outside $\varphi^{-1}(Z)$. Without loss of generality we assume that $P$ is not contained in $\varphi^{-1}(Z_1(B))$.  Let $F= g_i^{-1}(g_i(P))$ be the fiber of $g_i$ through $P$. Since no elliptic curve over $K$ has a $K$-point of order larger than $B$ by Theorem \ref{merel}, we conclude from Lemma \ref{fiberinfinite} that there is an irreducible component $C$ of $F$ for which $C(K)$ is infinite. From $\varphi(P) \not\in Z_0(B)$ we conclude that $\varphi(C)$ is not contained in $\mathcal{C}(B)$, so $C$ is a horizontal curve with respect to $g_2$ and $C$ is not contained in $\varphi^{-1}(T_2(B))$. 
Again by Theorem \ref{merel}, every point of $W(K)$ outside $\varphi^{-1}(S_2 \cup T_2(B))$ satisfies $\Xi_i(B)$, so 
by Lemma \ref{dense}, the set $W(K)$ is Zariski dense.
\end{proof}

\begin{proof}[Proof of Theorem \ref{mainone}]
Let $B$ be  as in \eqref{globalbound}.
Then by Proposition \ref{efficient} we may take $Z = Z_0(B) \cup (Z_1(B) \cap Z_2(B))$. 
\end{proof}

Recall that for any positive integer $N$, the curve $X_1(N)$
parametrizes pairs $(E,P)$, up to isomorphism over the algebraic closure of the ground field, 
of an elliptic curve $E$ and a point $P$ of order $N$.  The genus of $X_1(N)$ is at 
least $2$ for $N=13$ and $N \geq 16$ (see \cite[p. 109]{oggrat}). 
Let $\gamma_N \colon X_1(N) \to \AA_1(j)$ be the natural map to the $j$-line, sending $(E,P)$ to $j(E)$. 

\begin{lemma}\label{lemmaboundorderN}
Let $K$ be a field of characteristic zero with an element $j_0 \in K$. Let $N$ be a positive integer. 
Set $\mu(j_0)=4$ if $j_0=1728$, or $\mu(j_0)=6$ if $j_0=0$, or $\mu(j_0)=2$ otherwise. 
Let $E$ be an elliptic curve over $K$ with $j$-invariant $j_0$. Then the number of points in $E(K)$ of order $N$ is at most 
\begin{equation}\label{boundorderN}
\mu(j_0) \cdot \# \left( \gamma_N^{-1} (j_0) \cap X_1(N)(K) \right).
\end{equation}
\end{lemma}
\begin{proof}
Each point $P \in E(K)$ of order $N$ determines a point on $X_1(N)$ corresponding to the pair $(E,P)$, which maps under $\gamma_N$ to $j_0$. Two points $P,P' \in E(K)$ determine the same point on $X_1(N)$ if and only if there is an automorphism of $E$ that sends $P$ to $P'$.  As $E$ has only $\mu(j_0)$ automorphisms over $\overline{K}$, there are at most $\mu(j_0)$ points in $E(K)$ that determine a given point on $X_1(N)$. The lemma follows. 
\end{proof}

\begin{definition}
For any number field $K$ of degree $d$ over $\QQ$ we set 
$$
n_K = 2 \sum_{N=16}^B\left( \# X_1(N)(K) + \#\big(\gamma_N^{-1}(1728) \cap X_1(N)(K) \big) +
                                                                               2 \#\big(\gamma_N^{-1}(0) \cap X_1(N)(K) \big) \right),
$$
with $B$ as in \eqref{globalbound}.
\end{definition}

Note that $n_K$ is well defined for every number field $K$, as $X_1(N)(K)$ is finite for all $N \geq 16$ by Faltings' Theorem \cite{faltings}. Note also that $n_K$ equals the sum of \eqref{boundorderN} over all $j_0 \in K$ and all $N \in \{16, \ldots, B\}$. 

For $i=1,2$, let the $j$-maps $j_i \colon C_i \to \PP^1$ and their  ``degree" $d_i$ be as in Section \ref{settings}.

\begin{lemma}\label{limitedpoints}
Suppose $i \in \{1,2\}$ and let $K$ be any field extension of $k$.
Then there are at most $d_iM_in_K$ points in $W(K)$ outside $\varphi^{-1}(S_i)$ that do not satisfy $\Xi_i(15)$. 
\end{lemma}
\begin{proof}
Let $d$ be the degree of $K$ over $\QQ$ and let $B$ be as in \eqref{globalbound}. 
We know $X_1(N)(K)$ is empty for $N>B$ by Theorem \ref{merel}.
If we have $n_K=0$, then $X_1(N)(K)$ is empty for all $N>15$, so every point in $W(K)$ outside $\varphi^{-1}(S_i)$ 
satisfies $\Xi_i(15)$ and we are done, no matter what $d_iM_in_K$ is defined to be for $d_i = \infty$. Assume $n_K>0$. 
If $d_i = \infty$, then we are done, so we also assume 
$d_i < \infty$. Then the $j$-map $j_i \colon C_i$ and the induced $j$-map $j_i'  = j_i \circ \psi_i \colon D_i \to \PP^1$ are nonconstant of degree $d_i$. 
Let $\Gamma$ denote the set of all points $P\in W(K)$ that are not contained in 
$\varphi^{-1}(S_i)$ and that do not satisfy $\Xi_i(15)$.
Every point $P\in \Gamma$ lies on the smooth fiber $F=g_i^{-1}(t)$ above some $t\in D_i(K)$, where the divisor class of $(\beta_i(P)) - (P)$ has finite order $N$ in the Jacobian $J(F)$ of $F$ for some $N\geq 16$;  by Theorem \ref{merel} we have $N \leq B$. 
Summing over all $N \in \{16,\ldots, B\}$, over all $t\in D_i(K)$, and all points $Q$ of $J(g_i^{-1}(t))$ of order $N$ we find 
$$
\#\Gamma = \sum_{N=16}^B \mathop{{\sum}'}_{t \in D_i(K)} \mathop{\sum_{Q \in J(g_i^{-1}(t))}}_{{\rm order }\, Q=N}
\#\{P \in g_i^{-1}(t) \,\, : \,\, [(\beta_i(P))-(P)] = Q \},
$$
where the restricted sum is only over those $t\in D_i(K)$ for which $g_i^{-1}(t)$ is smooth. 
The summand is bounded by the degree of the map $F \to J(F), \, P \mapsto (P) - (\beta_i(P))$, with $F = g_i^{-1}(t)$, which equals the degree of the analogous map from the generic fiber of $g_i$ to its Jacobian; this generic map is geometrically equivalent to the map $\chi_i \colon \gV_i \to J(\gV_i)$, so the degree in question is $M_i$. 
By Lemma \ref{lemmaboundorderN}  the number of terms of the inner sum is bounded by \eqref{boundorderN} with $j_0=j(g_i^{-1}(t))=j_i'(t)$. For any $j_0\in K$ there are at most $d_i$ points $t\in D_i(K)$ with $j_i'(t)=j_0$, so grouping the points $t \in C_i(K)$ according to $j$-invariant, we find
$$
\#\Gamma \leq d_iM_i \sum_{N=16}^B  \sum_{j_0 \in K}  \mu(j_0) \cdot  \# \left( \gamma_N^{-1} (j_0) \cap X_1(N)(K) \right) = 
d_iM_in_K.
$$
\end{proof}


\begin{proposition}\label{notefficient}
Suppose the $j$-maps $j_1$ and $j_2$ are nonconstant.
Let $K$ be a finite extension of $k$. If $W(K)$ contains more than $n_K\cdot \min(d_1M_1,d_2M_2)$  points outside $\varphi^{-1}(Z)$ for $Z = Z_0(15) \cup Z_1(15) \cup Z_2(15)$, then $W(K)$ is dense in $W$. 
\end{proposition}
\begin{proof}
Without loss of generality we assume $d_1M_1\leq d_2M_2 < \infty$.  Suppose $W(K)$ contains more than $n_Kd_1M_1$ points outside $\varphi^{-1}(Z) \supset \varphi^{-1}(Z_1(15))$. Then by Lemma \ref{limitedpoints} there is such a point $P$ that either is contained in a singular fiber of $g_i$ or that satisfies $\Xi_i(15)$.  In either case Lemma \ref{fiberinfinite} implies that there is an 
irreducible component  $C$ of the fiber of $g_1$ through $P$ with $P\in C(K)$ for which $C(K)$ is infinite. 
From $\varphi(P) \not \in Z_0(15)$ we conclude $\varphi(C) \not \in \mathcal{C}(15)$, so $C$ is a horizontal curve with respect to $g_2$ and $C$ is not contained in $\varphi^{-1}(T_2(15))$.   By Lemma \ref{limitedpoints} only finitely many points $P \in W(K)$ outside 
$\varphi^{-1}(S_2)$ do not satisfy $\Xi_i(15)$, so by Lemma \ref{dense} the set $W(K)$ is dense in $W$.  
\end{proof}

\begin{proof}[Proof of Theorem \ref{maintwo}]
Then by Proposition \ref{notefficient} we may take $Z = Z_0(15) \cup Z_1(15) \cup Z_2(15)$. 
\end{proof}

The following proposition shows that we can take the set $Z$ much smaller, as long as we require the existence of more $K$-rational points outside $\varphi^{-1}(Z)$.

\begin{proposition}\label{efficienttwo}
Let $K$ be a finite extension of $k$. If $W(K)$ contains more than $n_K(d_1M_1+d_2M_2)$  points outside $\varphi^{-1}(Z)$ for $Z = Z_0(15) \cup (Z_1(15) \cap Z_2(15))$, then $W(K)$ is dense in $W$. 
\end{proposition}
\begin{proof}
Suppose $W(K)$ contains more than $n_K(d_1M_1+d_2M_2)$ points outside $\varphi^{-1}(Z)$. 
Then we have $d_iM_i < \infty$ for $i=1,2$,
and $W(K)$ contains either more than $d_1M_1n_K$ points outside $Z_0(15) \cup Z_1(15)$ or more than 
$d_2M_2n_K$ points outside $Z_0(15) \cup Z_2(15)$. Without loss of generality we assume the former case holds. 
Then by Lemma \ref{limitedpoints} there is such a point $P$ that either is contained in a singular fiber of $g_1$ or that satisfies $\Xi_i(15)$.  In either case Lemma \ref{fiberinfinite} implies that there is an 
irreducible component  $C$ of the fiber of $g_1$ through $P$ with $P\in C(K)$ for which $C(K)$ is infinite. 
From $\varphi(P) \not \in Z_0(15)$ we conclude $\varphi(C) \not \in \mathcal{C}(15)$, so $C$ is a horizontal curve with respect to $g_2$ and $C$ is not contained in $\varphi^{-1}(T_2(15))$.   By Lemma \ref{limitedpoints} only finitely many points $P \in W(K)$ outside 
$\varphi^{-1}(S_2)$ do not satisfy $\Xi_i(15)$, so by Lemma \ref{dense} the set $W(K)$ is dense in $W$.  
\end{proof}

%
%

\small
\nocite{*}
\bibliography{two_fibrations}
\bibliographystyle{plain}

\end{document}